\documentclass[12pt]{amsart}

\usepackage{amsmath,amssymb,epsfig}

\newtheorem{theorem}{Theorem}[section]

\newtheorem{corollary}[theorem]{Corollary}
\newtheorem{lemma}[theorem]{Lemma}

\numberwithin{equation}{section}

\newcommand{\R}{\mathbb{R}}
\newcommand{\N}{\mathbb{N}}

\newcommand{\C}{\mathbb{C}}
\newcommand{\T}{\mathbb{T}}
\newcommand{\D}{\mathbb{D}}
\newcommand{\DB}{\overline{\D}}
\newcommand{\RD}{\partial \D}
\newcommand{\ov}{\overline}
\newcommand{\dis}{\displaystyle}

\newcommand{\norm}[1]{\left\Vert#1\right\Vert}
\newcommand{\Notequiv}{/\kern-.6em\hbox{$\equiv$} }

\begin{document}

\rightline{\small\tt Final version, October 12, 2007}

\title[Inequalities for products of polynomials I]
{Inequalities for products of polynomials I}%
\author{I. E. Pritsker and S. Ruscheweyh}%

\address{Department of Mathematics, 401 Mathematical Sciences, Oklahoma State
University, Stillwater, OK 74078-1058, U.S.A.}%
\email{igor@math.okstate.edu}

\address{Institut f\"ur Mathematik, Universit\"at W\"urzburg, Am Hubland,
97074 W\"urzburg, Germany}%
\email{ruscheweyh@mathematik.uni-wuerzburg.de}

\thanks{Research of I.P. was partially supported by the National Security
Agency (grant H98230-06-1-0055), and by the Alexander von Humboldt
Foundation. S.R. acknowledges partial support from the
German-Israeli Foundation (grant G-809-234.6/2003), from FONDECYT
(grants 1040366 and 7040069) and from DGIP-UTFSM (grant 240104). }%
\subjclass[2000]{Primary 30C10; Secondary 30C85, 31A15}%
\keywords{Polynomials, products, factors, uniform norm, logarithmic
capacity, equilibrium measure, subharmonic function, Fekete points}%



\begin{abstract}

We study inequalities connecting the product of uniform norms of
polynomials with the norm of their product. This circle of problems
include the Gelfond-Mahler inequality for the unit disk and the
Kneser-Borwein inequality for the segment $[-1,1]$. Furthermore, the
asymptotically sharp constants are known for such inequalities over
arbitrary compact sets in the complex plane. It is shown here that
this best constant is smallest  (namely: 2) for a disk. We also
conjecture that it takes its largest value for a segment, among all
compact connected sets in the plane.

\end{abstract}

\maketitle


\section{The problem and its history} \label{sec1}

Let $E$ be a compact set in the complex plane ${\C}$.  For a function
$f:E\rightarrow \C$ define the
uniform (sup) norm as follows:
$$\| f \|_E = \sup_{z \in E} |f(z)|.$$
Clearly $\norm{f_1f_2}_E \le \norm{f_1}_E \norm{f_2}_E,$ but this
inequality is not reversible, in general, not even with a constant
factor in front of the right hand side. Indeed, $\norm{f_1}_E
\norm{f_2}_E \le C \norm{f_1f_2}_E$ does not hold for functions with
disjoint supports in $E$, for example. However, the situation is
quite different for algebraic polynomials $\{ p_k (z) \}_{k =1}^m$
and their product $p(z) := \prod_{k =1}^m p_k (z).$ Polynomial
inequalities of the form
\begin{equation} \label{1.1}
\prod_{k =1}^m \| p_k \|_E \leq C \| p \|_E,
\end{equation}
exist and are readily available.  One of the first results in this
direction is due to Kneser \cite{Kn}, for $E = [-1,1]$ and $m =2$
(see also Aumann \cite{Au}), who proved that
\begin{equation}  \label{1.2}
\| p_1 \|_{[-1,1]} \|p_2 \|_{[-1,1]} \leq K_{\ell ,n} \| p_1 p_2
\|_{[-1,1]}, \quad \deg p_1 = \ell,\ \deg p_2 =n-\ell,
\end{equation}
where
\begin{equation}  \label{1.3}
K_{\ell ,n} := 2^{n -1} \prod_{k =1}^{\ell} \left( 1 + \cos \frac{2k
-1}{2n} \pi \right) \prod_{k =1}^{n - \ell} \left( 1 + \cos \frac{2k
-1}{2n} \pi \right).
\end{equation}
Note that equality holds in (\ref{1.2}) for the Chebyshev
polynomial\\ $t(z) = \cos n \arccos z = p_1 (z) p_2 (z)$, with a
proper choice of the factors $p_1 (z)$ and $p_2 (z)$.  P. B. Borwein
\cite{Bor} generalized this to the multifactor inequality
\begin{equation} \label{1.4}
\prod_{k =1}^m \| p_k \|_{[-1,1]} \leq 2^{n -1} \prod_{k =1}^{[
\frac{n}{2} ]} \left( 1 + \cos \frac{2k -1}{2n} \pi \right)^2 \| p
\|_{[-1,1]}.
\end{equation}
He also showed that
\begin{equation} \label{1.5}
2^{n -1} \prod_{k =1}^{[ \frac{n}{2} ]} \left( 1 + \cos \frac{2k
-1}{2n} \pi \right)^2 \sim (3.20991 \ldots )^n \mbox{ as } n
\rightarrow \infty.
\end{equation}

A different version of inequality (\ref{1.1}) for $E = D$, where $D
:= \{ w: |w| \le 1 \}$ is the closed unit disk, was considered by
Gelfond \cite[p. 135]{Ge} in connection with the theory of
transcendental numbers:
\begin{equation} \label{1.6}
\prod_{k =1}^m \| p_k \|_{D} \leq e^n \| p \|_{D} .
\end{equation}
The latter inequality was improved by Mahler \cite{Ma1}, who
replaced $e$ by $2$:
\begin{equation} \label{1.7}
\prod_{k =1}^m \| p_k \|_{D} \leq 2^n \| p \|_{D} .
\end{equation}
It is easy to see that the base $2$ cannot be decreased, if $m =n$
and $n \rightarrow \infty$.  However, (\ref{1.7}) has recently been
further improved in two directions. D. W. Boyd \cite{Boy1, Boy2}
showed that, given the number of factors $m$ in (\ref{1.7}), one has
\begin{equation} \label{1.8}
\prod_{k =1}^m \| p_k \|_{D} \leq (C_m)^n \| p \|_{D},
\end{equation}
where
\begin{equation} \label{1.9}
C_m := \exp \left( \frac{m}{\pi} \int_0^{\pi/m} \log \left(2 \cos
\frac{t}{2}\right) dt \right)
\end{equation}
is asymptotically best possible for {\em each fixed} $m$, as $n
\rightarrow \infty$. Kro\'{o} and Pritsker \cite{KP} showed that,
for any $m \leq n,$
\begin{equation} \label{1.10}
\prod_{k =1}^m \| p_k \|_{D} \leq 2^{n -1} \| p \|_{D},
\end{equation}
where equality holds in (\ref{1.10}) for {\it each} $n \in {\N}$,
with $m =n$ and $p(z) = z^n -1$.

Inequalities (\ref{1.2})-(\ref{1.10}) clearly indicate that the
constant $C$ in (\ref{1.1}) grows exponentially fast with $n$, with
the base for the exponential depending on the set $E$. A natural
general problem arising here is to find the {\it smallest} constant
$M_E > 0,$ such that
\begin{equation} \label{1.11}
\prod_{k =1}^m \| p_k \|_E \leq M_E^n \| p \|_E
\end{equation}
for arbitrary algebraic polynomials $\{ p_k (z) \}_{k =1}^m$ with
complex coefficients, where $p(z) = \prod_{k =1}^m p_k (z)$ and $n =
\deg p$. The solution of this problem is based on the logarithmic
potential theory (cf. \cite{Ts} and \cite{Ra}). Let ${\rm cap}(E)$
be the {\it logarithmic capacity} of a compact set $E \subset {\C}$.
For $E$ with ${\rm cap}(E)>0$, denote the {\it equilibrium measure}
of $E$  by $\mu_E$. We remark that $\mu_E$ is a positive unit Borel
measure supported on $\partial E$ (see \cite[p. 55]{Ts}). Define
\begin{equation} \label{1.12}
d_E(z) := \max_{t \in E} |z -t|, \qquad z \in {\C},
\end{equation}
which is clearly a positive and continuous function in ${\C}$. It is
easy to see that the logarithm of this distance function is
subharmonic in $\C.$ Furthermore, it has the following integral
representation
\[
\log d_E(z) = \int \log |z -t| d \sigma_E(t), \quad z \in {\C} ,
\]
where $\sigma_E$ is a positive unit Borel measure in ${\C}$ with
unbounded  support, see Lemma 5.1 of \cite{Pr1} and \cite{LP01}. For
further in-depth analysis of the representing measure $\sigma_E$, we
refer to the recent paper of Gardiner and Netuka \cite{GN}. This
integral representation is the key fact used by the first author to
prove the following result \cite{Pr1}.

\begin{theorem} \label{thm1.1}
Let $E \subset {\C}$ be a compact set, ${\rm cap} (E) >0$.  Then the
best constant $M_E$ in {\rm (\ref{1.11})} is given by
\begin{equation} \label{1.13}
M_E = \frac{\exp\left(\dis\int \log d_E(z) d \mu_E (z)\right)}{{\rm
cap} (E)} .
\end{equation}
\end{theorem}

Theorem \ref{thm1.1} is applicable to any compact set with a
connected component consisting of more than one point (cf. \cite[p.
56]{Ts}). In particular, if $E$ is a continuum, i.e., a connected
set, then we obtain a simple universal bound for $M_E$  \cite{Pr1}:

\begin{corollary} \label{cor1.2}
Let $E \subset \C$ be a bounded continuum (not a single point). Then
we have
\begin{equation} \label{1.14}
M_E \leq \frac{{\rm diam}(E)}{{\rm cap} (E)} \leq 4,
\end{equation}
where ${\rm diam}(E)$ is the Euclidean diameter of the set $E$.
\end{corollary}

On the other hand, for non-connected sets $E$ the constants $M_E$ can be
arbitrarily large. For
example, consider $E_k=[-\sqrt{k+4},-\sqrt{k}] \cup
[\sqrt{k},\sqrt{k+4}]$, so that cap$(E_k)=1$ \cite{Ra} and
\[
M_E = \exp\left(\int \log d_{E_k}(z)\,d\mu_{E_k}(z)\right) \ge
e^{\log (2\sqrt{k})} \to \infty\quad \mbox{as } k\to\infty.
\]

For the closed unit disk $D$, we have that ${\rm cap}(D) =1$
\cite[p. 84]{Ts} and that
\begin{equation} \label{1.15}
d\mu_{D} = \frac{d\theta}{2 \pi},
\end{equation}
where $d \theta$ is the arclength on $\partial D$.  Thus Theorem
\ref{thm1.1} yields
\begin{equation} \label{1.16}
M_{D} = \exp\left(\frac{1}{2 \pi} \int_0^{2 \pi} \log d_{D}(e^{i
\theta})\ d \theta\right) = \exp\left(\frac{1}{2 \pi} \int_0^{2 \pi}
\log 2\ d \theta\right) =2,
\end{equation}
so that we immediately obtain Mahler's inequality (\ref{1.7}).

If $E = [-1,1]$ then ${\rm cap}([-1,1]) = 1/2$ and
\begin{equation} \label{1.17}
d\mu_{[-1,1]} = \frac{dx}{\pi \sqrt{1 -x^2}} , \quad x \in [-1,1],
\end{equation}
which is the Chebyshev (or arcsin) distribution (see \cite[p.
84]{Ts}). Using Theorem \ref{thm1.1}, we obtain
\begin{eqnarray} \label{1.18}
M_{[-1,1]} & = & 2\exp\left(\frac{1}{\pi} \int_{-1}^1 \frac{\log
d_{[-1,1]}(x)}{\sqrt{1 -x^2}} dx\right) = 2\exp\left(\frac{2}{\pi}
\int_0^1 \frac{\log (1 +x)}{\sqrt{1 -x^2}} dx \right) \nonumber \\
& = & 2\exp\left(\frac{2}{\pi} \int_0^{\pi/2} \log (1 + \sin t) dt
\right) \approx 3.2099123,
\end{eqnarray}
which gives the asymptotic version of Borwein's inequality
(\ref{1.4})-(\ref{1.5}).

Considering the above analysis of Theorem \ref{thm1.1}, it is
natural to conjecture that the sharp universal bounds for $M_E$ are
given by
\begin{equation} \label{1.19}
2=M_{D} \le M_E \le M_{[-1,1]} \approx 3.2099123,
\end{equation}
for any bounded non-degenerate continuum $E$, see \cite{Pr3}.

It follows directly from the definition that $M_E$ is invariant with
respect to the similarity transformations of the plane. Thus  we can
normalize the problem by setting ${\rm cap} (E)=1$. Thus,
equivalently, we want to find the maximum and the minimum of the
functional
\begin{equation} \label{1.20}
\tau(E):=\int \log d_E(z) d \mu_E (z)
\end{equation}
over all compact connected sets $E$ in the plane satisfying the
above normalization. These questions are addressed in Section 2 of
the paper. Section 3 discusses a more refined version of our problem
on the best constant in \eqref{1.1}. All proofs are given in Section
4.

In the forthcoming paper \cite{PR}, we consider various improved
bounds of the constant $M_E$, e.g., bounds for rotationally
symmetric sets. From a different perspective, the results of Boyd
\eqref{1.8}-\eqref{1.9} suggest that for some sets the constant
$M_E$ can be replaced by a smaller one, if the number of factors is
fixed. We characterize such sets in \cite{PR}, and find the improved
constant.

\medskip
The problems considered in this paper have many applications in
analysis, number theory and computational mathematics. We mention
specifically applications in transcendence theory (see Gelfond
\cite{Ge}), and in designing algorithms for factoring polynomials
(see Boyd \cite{Boy3} and Landau \cite{La}). A survey of the results
involving norms different from the sup norm (e.g., Bombieri norms)
can be found in \cite{Boy3}. For polynomials in several variables,
see the results of Mahler \cite{Ma2} for the polydisk, of Avanissian
and Mignotte \cite{AM} for the unit ball in ${\C}^k$. Also, see
Beauzamy and Enflo \cite{BeEn}, and Beauzamy, Bombieri, Enflo and
Montgomery \cite{BBEM} for multivariate polynomials in different
norms.

\medskip
{\bf Acknowledgements.} The authors wish to express their gratitude
to Richard Laugesen for several helpful discussions about these
problems. Alexander Solynin communicated to the first author a
sketch of proof for the inequality $M_E \ge 2$ for connected sets.
We would like to thank him for the kind permission to use his
argument in the proof of Theorem \ref{thm2.5}. This paper was
written while the first author was visiting the University of
W\"urzburg as a Humboldt Foundation Fellow. He would like to thank
the Department of Mathematics and the Function Theory research group
for their hospitality.

\section{Sharp bounds for the constant $M_E$} \label{sec2}

We study bounds for the constant $M_E$ in this section, where
$E\subset\C$ is a compact set satisfying ${\rm cap} (E)>0.$ Our main
goal here is to prove \eqref{1.19}. It is convenient to first give
some general observations on the properties of $M_E$.

\begin{theorem} \label{thm2.1}
Let $I \subset E$ be compact sets in $\C$, ${\rm cap} (I) >0$.
Denote the unbounded components of $\ov\C\setminus E$ and
$\ov\C\setminus I$ by $\Omega_E$ and $\Omega_I$. If $d_E(z)=d_I(z)$
for all $z \in \partial \Omega_I$ then $M_E \le M_I,$ with equality
holding only when $\textup{cap}(\Omega_I\setminus \Omega_E)=0.$
\end{theorem}

This theorem gives several interesting consequences. In particular,
we show that if the set $E$ is contained in a disk whose diameter
coincides with the diameter of $E$ then its constant $M_E$ does not
exceed that of a segment. Thus segments indeed maximize $M_E$ among
such sets. Denote the closed disk of radius $r$ centered at $z$ by
$D(z,r).$

\begin{corollary} \label{cor2.2}
Let $z,w\in E$ satisfy $\textup{diam}\,E=|z-w|$ and $[z,w]\subset
E.$ If $E\subset
D\left(\frac{z+w}{2},\frac{\textup{diam}\,E}{2}\right)$ then $M_E
\le M_{[z,w]} = M_{[-2,2]}.$
\end{corollary}

The next results shows that the constant decreases when the set is
enlarged in a certain way.

\begin{corollary} \label{cor2.3}
Let $E^*:=\bigcap_{z\in \partial\Omega_E} D(z,d_E(z))$, where $E
\subset {\C}$ is compact, ${\rm cap} (E) >0$. If $H$ is a compact
set such that $E\subset H\subset E^*,$ then $M_H \le M_E.$ Equality
holds if and only if $\textup{cap}(\Omega_E\setminus \Omega_H)=0.$
\end{corollary}

Let conv$(H)$ be the convex hull of $H$. The operation of taking the
convex hull of a set satisfies the assumption of Corollary
\ref{cor2.3} (or Theorem \ref{thm2.1}), which gives

\begin{corollary} \label{cor2.4}
Let $V\subset\C$ be a compact set, ${\rm cap}(V)>0.$ If
$H:=\ov\C\setminus\Omega_V$ is not convex, then $M_{{\rm conv}(H)} <
M_H.$
\end{corollary}

The above results help us to show that the minimum of $M_E$ is
attained for the closed unit disk $D,$ among all sets of positive
capacity (connected or otherwise).

\begin{theorem} \label{thm2.5}
Let $E \subset {\C}$ be an arbitrary compact set, ${\rm cap} (E)
>0$. Then $M_E\ge 2,$ where equality holds if and only if
$\ov\C\setminus\Omega_E$ is a closed disk.
\end{theorem}
In other words, $M_E=2$ only for sets whose polynomial convex hull
is a disk. This may also be described by saying that $M_E = 2$ if
and only if $\partial U \subset E \subset U$, where $U$ is a closed
disk.

Proving that the maximum of $M_E$ for {\em arbitrary} continua is
attained for a segment is a more difficult problem. In fact, it is
related to some old open problems on the moments of the equilibrium
measure (or circular means of conformal maps), see P\'olya and
Schiffer \cite{PS}, and Pommerenke \cite{Po}. In particular, we use
the results of \cite{PS} and \cite{Po} to show that

\begin{theorem} \label{thm2.6}
Let $E \subset {\C}$ be a connected compact set, ${\rm cap} (E)
>0$.\\
\textup{(i)} If the center of mass $c:=\int z\,d\mu_E(z)$ for
$\mu_E$ belongs to $E$, then
\begin{align} \label{2.2}
M_E  < 2+4.02/\pi \approx 3.279606.
\end{align}
\textup{(ii)} If $E$ is convex then
\begin{align} \label{2.3}
M_E  < 2+4/\pi \approx 3.27324.
\end{align}
\end{theorem}
This should be compared with $M_{[-2,2]} = M_{[-1,1]} \approx
3.2099123.$

After this paper had been written, a new related manuscript
\cite{BLP} appeared. That manuscript contains a proof of our
conjecture $M_E \le M_{[-2,2]}$ for centrally symmetric continua, as
well as another quite general conjecture (if true) implying $M_E \le
M_{[-2,2]}$ holds for all continua.

\section{Refined problem} \label{sec5}

The constant $M_E$ represents the base of rather crude exponential
asymptotic for the constant in inequality \eqref{1.1}. A more
refined question is to find the sharp constant attained with
equality. Such constants are known in the case of a segment, see
\eqref{1.4} and \cite{Bor}; and in the case of a disk, see
\eqref{1.10} and \cite{KP}. Let $E$ be any compact set in the plane,
and let $\prod_{k=1}^m p_k(z) = \prod_{j=1}^n (z-z_j),$ where
$p_k(z)$ are arbitrary monic polynomials with complex coefficients.
Define the constant
\begin{align} \label{5.1}
C_E(n) := \sup_{p_k} \frac{\dis\prod_{k=1}^m \norm{p_k}_E}
{\norm{\dis\prod_{k=1}^m p_k}_E} = \sup_{z_j\in\C}
\frac{\dis\prod_{j=1}^n \norm{z-z_j}_E} {\norm{\dis\prod_{j=1}^n
(z-z_j)}_E}.
\end{align}
If cap$(E)>0$ then it follows from Theorem \ref{thm1.1} that $1 \le
C_E(n) \le M_E^n.$ The refined version of our conjecture in
\eqref{1.19} is as follows:
\begin{align} \label{5.2}
2^{n-1}=C_D(n) \le C_E(n) \le C_{[-2,2]}(n) = 2^{n -1} \prod_{k
=1}^{[n/2]} \left( 1 + \cos \frac{2k -1}{2n} \pi \right)^2
\end{align}
for any connected compact set $E$ of positive capacity.

\section{Proofs} \label{sec6}

\begin{proof}[Proof of Theorem \ref{thm2.1}]

Since $I\subset E$, we have that $\textup{cap}(E)\ge\textup{cap}(I)
> 0$. Let $g_E(z,\infty)$ and $g_I(z,\infty)$ be the Green's
functions for $\Omega_E$ and $\Omega_I$, with poles in infinity. We
follow the standard convention by setting $g_E(z,\infty)=0,\
z\not\in\ov\Omega_E$ and $g_I(z,\infty)=0,\ z\not\in\ov\Omega_I.$ It
follows from the maximum principle that $g_E(z,\infty) \le
g_I(z,\infty)$ for all $z\in\C.$ Furthermore, this inequality is
strict in $\Omega_E$, unless $\textup{cap}(\Omega_I\setminus
\Omega_E)=0.$

Using the integral representation for $d_E(z)$ from Lemma 5.1 of
\cite{Pr1} (see also \cite{LP01} and \cite{GN}) and the Fubini
theorem, we obtain that
\begin{align*}
\log M_E &= \int \log d_E(z)\,d\mu_E(z) - \log \textup{cap}(E) \\ &=
\int \int \log |z-t|\,d\sigma_E(t) d\mu_E(z) - \log \textup{cap}(E)
\\ &= \int \left( \int \log |z-t|\,d\mu_E(z) - \log \textup{cap}(E)
\right) d\sigma_E(t) = \int g_E(t,\infty)\,d\sigma_E(t),
\end{align*}
where the last equality follows from the well known identity\\
$g_E(t,\infty)=\int \log |z-t|\,d\mu_E(z) - \log \textup{cap}(E)$
\cite{Ra}. It is clear that
\[
\int g_E(t,\infty)\,d\sigma_E(t) \le \int
g_I(t,\infty)\,d\sigma_E(t),
\]
with equality possible if and only if
$\textup{cap}(\Omega_I\setminus \Omega_E)=0.$ Indeed, if we have
equality in the above inequality, then $g_E(z,\infty) =
g_I(z,\infty)$ for all $z\in\textup{supp}\,\sigma_E.$ But
$\textup{supp}\,\sigma_E$ is unbounded, so that $g_E(z,\infty) =
g_I(z,\infty)$ in $\Omega_E$ by the maximum principle. Hence we
obtain that
\begin{align*}
\log M_E &\le \int g_I(t,\infty)\,d\sigma_E(t) = \int \left( \int
\log |z-t|\,d\mu_I(z) - \log \textup{cap}(I) \right) d\sigma_E(t)
\\ &= \int \log d_E(z)\,d\mu_I(z) - \log \textup{cap}(I) =
\int \log d_I(z)\,d\mu_I(z) - \log \textup{cap}(I) \\ &= \log M_I,
\end{align*}
with equality if and only if $\textup{cap}(\Omega_I \setminus
\Omega_E)=0.$ Note that we used supp$\,\mu_I \subset
\partial\Omega_I,$ so that $d_E(z)=d_I(z)$ for
$z\in \textup{supp}\,\mu_I.$

\end{proof}

\begin{proof}[Proof of Corollary \ref{cor2.2}]

Let $I=[z,w]$ be the segment connecting the points $z$ and $w$,
i.e., the common diameter of $E$ and the disk containing it. Observe
that we have $d_E(t)=d_I(t)$ for all $t\in\partial\Omega_I=I$ under
the stated geometric conditions. Since all assumptions of Theorem
\ref{thm2.1} are satisfied, we obtain that $M_E\le M_{[z,w]} =
M_{[-2,2]},$ where the last equality follows from the invariance
with respect to the similarity transformations of the plane.

\end{proof}

\begin{proof}[Proof of Corollary \ref{cor2.3}]

Observe that $E\subset D(z,d_E(z))$ for any $z\in\C.$ Hence
$E\subset E^*.$ Since $E\subset H \subset E^*,$ we immediately
obtain that $d_E(z) \le d_H(z) \le d_{E^*}(z),\ z\in\C.$ On the
other hand, the definition of $E^*$ gives that $d_E(z) = d_{E^*}(z)$
for all $z\in\partial\Omega_E.$ Therefore $d_E(z) = d_{H}(z)$ for
all $z\in\partial\Omega_E,$ and the result follows from Theorem
\ref{thm2.1}.

\end{proof}

\begin{proof}[Proof of Corollary \ref{cor2.4}]

We apply Theorem \ref{thm2.1} again, with $I=H$ and
$E=\textup{conv}(H).$ It was shown in \cite{LP01} that $d_H(z) =
d_{\textup{conv}(H)}(z)$ for all $z\in\C$, where $H$ is an arbitrary
compact set. Since $H$ is not convex in our case, we obtain that
$\textup{cap}(\Omega_I\setminus \Omega_E)>0$ and $M_E < M_I.$

\end{proof}

For the proof of Theorem~\ref{thm2.5} we need a special case of the
following lemma, which may be of some independent interest. Let
 $\Delta:=\{w:|w|>1\}$, and $\D:=\{z: |z|<1\}$ the unit disk.

\begin{lemma} \label{lem6.1}
Let $\Gamma$ be a Jordan domain and let $\Psi(z):=
cw+\sum_{k=0}^{\infty}a_kw^{-k}$ be a conformal map of $\Delta$ onto
$\Omega_\Gamma$. Furthermore assume that
\begin{equation}
  \label{eq:1}
  \forall x,z\in\partial\Delta:\quad |\Psi(z)-\Psi(x)|\leq |\Psi(z)-\Psi(-z)|.
\end{equation}
Then $\Gamma$ is a disk.
\end{lemma}

\begin{proof}
  First note that by Carath\'eodory's theorem \cite[p. 18]{Po92} $\Psi$
extends to a homeomorphism  of $\overline{\Delta}$, so that \eqref{eq:1}
makes
sense. Also there is no loss of generality in assuming $0\in\Gamma$,
so that $\Psi(z)\neq 0$ in $\overline{\Delta}$. Let
$$
g(z):=\frac{1}{\Psi(1/z)},\quad z\in\DB.
$$
Then $g(z)=z/c+\sum_{k=2}^{\infty}b_kz^k$ is a homeomorphism of $\DB$
onto the closure of the Jordan domain $\Gamma^*$, the interior
domain of the Jordan curve $1/\partial\Gamma$. Note that $g(0)=0,
g'(0)=1/c\neq0$.

Let $1/z\in\RD$, and in \eqref{eq:1} we replace $1/x\in\RD$ by $-1/xz$ which
is also in $\RD$. Condition \eqref{eq:1} then becomes
$$
1\geq\left|
  \frac{\frac{1}{g(z)}-\frac{1}{g(-xz)}}{\frac{1}{g(z)}-\frac{1}{g(-z)}}
\right|=
\left|
  \frac{xg(-z)}{g(-xz)}\frac{g(-xz)-g(z)}{g(-z)-g(z)}
\right|,\quad x,z\in\RD.
$$
Note that the function
$$
F(x,z):=\frac{xg(-z)}{g(-xz)}\frac{g(-xz)-g(z)}{g(-z)-g(z)}
$$
is analytic in $(x,z)\in\D^2$, and by the maximum principle, applied
to both variables separately, we find that
$$
|F(x,z)|\leq1, \quad x,z\in\DB.
$$
Now fix $z_0$ with $0<|z_0|<1$. Then $x\mapsto F(x,z_0)$ is analytic
in $\DB$, satisfies $ |F(x,z_0)|\leq1$ for $x\in\DB$, and, in
addition,  $F(1,z_0)=1$. The Julia-Wolf Lemma \cite[p. 82]{Po92}
then says that $F'(1,z_0)>0$, or
$$
1+\frac{-z_0g'(-z_0)}{g(-z_0)}\frac{g(z_0)}{g(-z_0)-g(z_0)}>0.
$$
Obviously this must be true for any $z_0$, and so, by the identity
principle, we are left with the relation
$$
\frac{-zg'(-z)}{g(-z)}\frac{g(z)}{g(-z)-g(z)}\equiv\alpha,\quad
  z\in\D,
$$
where $\alpha>-1$ is some real constant. Letting $z\rightarrow 0$,
we find $\alpha=-\frac{1}{2}$. Hence we are left with the
difference-differential equation
\begin{equation}
  \label{eq:2}
  \frac{zg'(z)}{g(z)}\frac{g(-z)}{g(-z)-g(z)}=\frac{1}{2},\quad
  z\in\D.
\end{equation}
In terms of $\Psi$ this reads
$$
2w\Psi'(w)=\Psi(w)-\Psi(-w),\quad w\in \Omega_\Gamma.
$$
From this we conclude that $w\Psi'(w)$ is an odd function, which, in
turn, implies that $\Phi(w):=\Psi(w)-a_0$ is odd as well. For $\Phi$
we then get the equation $w\Phi'(w)=\Phi(w)$, or $\Phi(w)=cw$. This
implies $\Psi(w)=cw+a_0$ and therefore that $\Gamma$ is a disk.
\end{proof}

\begin{proof}[Proof of Theorem \ref{thm2.5}]

Note that for any compact set $E$, we have $M_E=M_W$, where
$W:=\ov\C\setminus\Omega_E$. This follows because $\mu_E=\mu_W$
\cite{Ra} and $d_E(z) = d_W(z), \ z\in\C.$ Corollary \ref{cor2.4}
now implies that
\[
\inf\{ M_E: E \mbox{ is compact}\} = \inf\{ M_H: H \mbox{ is convex
and compact}\}.
\]
Hence we can assume that $E$ is convex from the start. We also set
cap$(E)=1,$ because $M_E$ is invariant under similarity transforms.
Thus $\partial E$ is a rectifiable Jordan curve (or a segment when
$E=\partial E$). The following argument that shows $M_E \ge 2$ for
all connected sets is due to A. Solynin. Let
$\Psi:\Delta\to\Omega_E$ be the standard conformal map:
\[
\Psi(w)=w+a_0+\sum_{k=1}^{\infty} \frac{a_k}{w^k},\qquad w\in\Delta.
\]
Recall that $\Psi$ can be extended as a homeomorphism of $\ov\Delta$
onto $\ov\Omega_E,$ with $\Psi(\T)=\partial E,\ \T:=\partial\Delta.$
It is clear that
\[
d_E(\Psi(e^{it})) \ge |\Psi(e^{it})-\Psi(-e^{it})|,\qquad
t\in[0,2\pi).
\]
Since $\Psi(w)$ is univalent in $\Delta,$ the function
\[
H(w):=\frac{\Psi(w)-\Psi(-w)}{w}
\]
is analytic and non-vanishing in $\Delta$, including $w=\infty.$
Furthermore, $H(\infty):=\dis\lim_{w\to\infty} H(w) = 2.$ It follows
that $h(w):=\log |H(w)|$ is harmonic in $\Delta.$ Recall that the
equilibrium measure $\mu_E$ is the harmonic measure of $\Omega_E$ at
$\infty,$ which is invariant under the conformal transformation
$\Psi,$ see \cite{Ra}. Hence
\begin{align*}
\log M_E &= \int \log d_E(z)\,d\mu_E(z) = \frac{1}{2\pi}
\int_0^{2\pi} \log d_E(\Psi(e^{it}))\,dt \\ &\ge \frac{1}{2\pi}
\int_0^{2\pi} \log \left| \frac{\Psi(e^{it})-\Psi(-e^{it})}{e^{it}}
\right|\,dt = \log 2,
\end{align*}
where we used the Mean Value Theorem for $h(w)$ on the last step.
Thus we conclude that $M_E \ge 2=M_D$ holds for all compact sets
$E.$

Recall that $M_E=M_W$, where $W=\ov\C\setminus\Omega_E$. If $M_E=2$
then $M_W=2$, so that $W$ must be convex by Corollary \ref{cor2.4}.
Since $M_W>3.2$ for any segment, we have that $W$ is the closure of
a convex domain. We can assume that cap$(W)=1$ after a dilation.
Repeating the above argument for $W$ instead of $E$, we obtain that
\begin{align*}
\log 2 &= \log M_W = \frac{1}{2\pi} \int_0^{2\pi} \log
d_W(\Psi(e^{it}))\,dt \\ &\ge \frac{1}{2\pi} \int_0^{2\pi} \log
\left| \Psi(e^{it})-\Psi(-e^{it}) \right|\,dt = \log 2.
\end{align*}
It follows that
\[
\int_0^{2\pi} \left( \log d_W(\Psi(e^{it})) - \log \left|
\Psi(e^{it})-\Psi(-e^{it}) \right| \right)\,dt = 0,
\]
and that $d_W(\Psi(e^{it})) = \left| \Psi(e^{it})-\Psi(-e^{it})
\right|$ a.e. on $[0,2\pi).$ But these functions are clearly
continuous, so that
\[
d_W(\Psi(e^{it})) = \left| \Psi(e^{it})-\Psi(-e^{it}) \right| \quad
\forall t\in\R.
\]
An application of Lemma~\ref{lem6.1} with $\Gamma$ the interior
domain of $W$ shows that $W$ must be a disk. We would also like to
mention that A. Solynin obtained a different proof of the fact that
$M_E=2$ for a connected set $E$ implies $W$ is a disk.
\end{proof}

\begin{proof}[Proof of Theorem \ref{thm2.6}]

Recall that $M_E$ is invariant under similarity transformations.
Hence we can assume again that cap$(E)=1$ and $\dis \int z \,
d\mu_E(z) = 0.$ The latter condition means that the center of mass
for the equilibrium measure is at the origin. If we introduce the
conformal map $\Psi:\Delta\to\Omega_E$, as in the previous proof,
then this condition translates into $a_0=0$, i.e.,
\[
\Psi(w)=w+\sum_{k=1}^{\infty} \frac{a_k}{w^k},\qquad w\in\Delta.
\]
Theorem 1.4 of \cite[p. 19]{Po75} gives that $E\subset D(0,2)$, so
that $d_E(z) \le 2 + |z|, \ z\in E,$ by the triangle inequality.
Note that this is sharp for $E=[-2,2]$. Applying Jensen's
inequality, we have
\[
\log M_E = \int \log d_E(z)\,d\mu_E(z) \le \int \log(2+|z|)\, d
\mu_E (z) < \log\left(2 + \int |z|\, d \mu_E (z)\right).
\]
Estimates \eqref{2.2} and \eqref{2.3} now follow from the results of
Pommerenke \cite{Po}, and of P\'olya and Schiffer \cite{PS}, who
estimated the integral
\[
\int |z|\, d \mu_E (z) = \frac{1}{2\pi} \int_0^{2\pi}
|\Psi(e^{it})|\,dt < 4.02/\pi \quad(\mbox{or }\le 4/\pi),
\]
under the corresponding assumptions.

\end{proof}


\begin{thebibliography}{11}

\bibitem{Au}
G. Aumann, {\it Satz \"{u}ber das Verhalten von
Polynomen auf Kontinuen}, Sitz. Preuss. Akad. Wiss. Phys.-Math. Kl.
(1933), 926-931.
\bibitem{AM}
V. Avanissian and M. Mignotte, {\it A variant of an inequality of
Gel'fond and Mahler}, Bull. London Math. Soc. {\bf 26} (1994),
64-68.
\bibitem{BLP}
A. Baernstein II, R. S. Laugesen, and I. E. Pritsker, {\it Moment
inequalities for equilibrium measures}, manuscript.
\bibitem{BBEM}
B. Beauzamy, E. Bombieri, P. Enflo and H. L. Montgomery, {\it
Products of polynomials in many variables}, J. Number Theory {\bf
36} (1990), 219-245.
\bibitem{BeEn}
B. Beauzamy and P. Enflo, {\it Estimations de produits de
polyn\^{o}mes}, J. Number Theory {\bf 21} (1985), 390-413.
\bibitem{BST}
C. Benitez, Y. Sarantopoulos and A. Tonge, {\it Lower bounds for
norms of products of polynomials}, Math. Proc. Cambridge Philos.
Soc. {\bf 124} (1998), 395--408.
\bibitem{Bor}
P. B. Borwein, {\it Exact inequalities for the norms of factors of
polynomials}, Can. J. Math. {\bf 46} (1994), 687-698.
\bibitem{BE}
P. Borwein and T. Erd\'{e}lyi, Polynomials and Polynomial
Inequalities, Springer-Verlag, New York, 1995.
\bibitem{Boy1}
D. W. Boyd, {\it Two sharp inequalities for the norm of a factor of
a polynomial}, Mathematika {\bf 39} (1992), 341-349.
\bibitem{Boy2}
D. W. Boyd, {\it Sharp inequalities for the product of polynomials},
Bull. London Math. Soc. {\bf 26} (1994), 449-454.
\bibitem{Boy3}
D. W. Boyd, {\it Large factors of small polynomials}, Contemp. Math.
{\bf 166} (1994), 301-308.
\bibitem{Boy4}
D. W. Boyd, {\it Bounds for the height of a factor of a polynomial
in terms of Bombieri's norms: I. The largest factor}, J. Symbolic
Comp. {\bf 16} (1993), 115-130.
\bibitem{Boy5}
D. W. Boyd, {\it Bounds for the height of a factor of a polynomial
in terms of Bombieri's norms: II. The smallest factor}, J. Symbolic
Comp. {\bf 16} (1993), 131-145.
\bibitem{GN}
S. J. Gardiner and I. Netuka, {\it Potential theory of the
farthest-point distance function}, J. Anal. Math. {\bf 101} (2006),
163-177.
\bibitem{Ge}
A. O. Gelfond, Transcendental and Algebraic Numbers, Dover, New
York, 1960.
\bibitem{Gl}
P. Glesser, {\it Nouvelle majoration de la norme des facteurs d'un
polyn\^{o}me}, C. R. Math. Rep. Acad. Sci. Canada {\bf 12} (1990),
224-228.
\bibitem{Go69}
G. M. Goluzin, Geometric Theory of Functions of a
Complex Variable, Vol. 26 of Translations of Mathematical
Monographs, Amer. Math. Soc., Providence, R.I., 1969.
\bibitem{Gr}
A. Granville, {\it Bounding the coefficients of a divisor of a given
polynomial}, Monatsh. Math. {\bf 109} (1990), 271-277.
\bibitem{Kn}
H. Kneser, {\it Das Maximum des Produkts zweies Polynome}, Sitz.
Preuss. Akad. Wiss. Phys.-Math. Kl. (1934), 429-431.
\bibitem{KP}
A. Kro\'{o} and I. E. Pritsker, {\it A sharp version of Mahler's
inequality for products of polynomials}, Bull. London Math. Soc.
{\bf 31} (1999), 269-278.
\bibitem{La}
S. Landau, {\it Factoring polynomials quickly}, Notices Amer. Math.
Soc. {\bf 34} (1987), 3-8.
\bibitem{LP01}
R. S. Laugesen and I. E. Pritsker, {\it Potential theory
of the farthest-point distance function}, Can. Math. Bull. {\bf 46}
(2003), 373-387.
\bibitem{Ma1}
K. Mahler, {\it An application of Jensen's formula to polynomials},
Mathematika {\bf 7} (1960), 98-100.
\bibitem{Ma2}
K. Mahler, {\it On some inequalities for polynomials in several
variables}, J. London Math. Soc. {\bf 37} (1962), 341-344.
\bibitem{Mi}
M. Mignotte, {\it Some useful bounds}, In ``Computer Algebra,
Symbolic and Algebraic Computation" (B. Buchberger et al., eds.),
pp. 259-263, Springer-Verlag, New York, 1982.
\bibitem{Ne}
Z. Nehari, Conformal Mapping, McGraw-Hill Co, New York, 1952.
\bibitem{PS}
G. P\'olya and M. Schiffer, {\it Sur la repr\'esentation conforme de
l'ext\'erieur d'une courbe ferme\'e convexe}, C. R. Acad. Sci.
(Paris) {\bf 248} (1959), 2837-2839.
\bibitem{Po}
Ch. Pommerenke, {\it \"Uber einige Klassen meromorpher schichter
Funktionen}, Math. Zeit. {\bf 78} (1962), 263-284.
\bibitem{Po75}
Ch. Pommerenke, Univalent Functions, Vandenhoeck \& Ruprecht,
G\"ottingen, 1975.
\bibitem{Po92}
Ch. Pommerenke, Boundary Behaviour of Conformal Maps, Springer,
Berlin-Heidelberg-New York, 1992.
\bibitem{Pr1}
I. E. Pritsker, {\it Products of polynomials in uniform norms},
Trans. Amer. Math. Soc. {\bf 353} (2001), 3971-3993.
\bibitem{Pr2}
I. E. Pritsker, {\it An inequality for the norm of a polynomial
factor}, Proc. Amer. Math. Soc. {\bf 129} (2001), 2283-2291.
\bibitem{Pr3}
I. E. Pritsker, {\it Norms of products and factors of polynomials},
in ``Number Theory for the Millennium III," M. A. Bennett, B. C.
Berndt, N. Boston, H. Diamond, A. J. Hildebrand and W. Philipp
(eds.), pp. 173-189, A K Peters, Ltd., Natick, 2002.
\bibitem{PR}
I. E. Pritsker and S. Ruscheweyh, {\em Inequalities for products of
polynomials II}, manuscript.
\bibitem{Ra}
T. Ransford, Potential Theory in the Complex Plane,
Cambridge University Press, Cambridge, 1995.
\bibitem{Ts}
M. Tsuji, Potential Theory in Modern Function Theory, Chelsea Publ. Co.,
New York, 1975.
\bibitem{We}
R. Webster, Convexity, Oxford Univ. Press, Oxford, 1994.

\end{thebibliography}
\end{document}